\documentclass[oneside]{amsart}

\numberwithin{equation}{section}
\newcommand{\be}{\begin{equation}}
\newcommand{\ee}{\end{equation}}

\allowdisplaybreaks

\newcommand{\R}{\mathbb R}
\newcommand{\Z}{\mathbb Z}
\newcommand{\N}{\mathbb N}

\newcommand{\ep}{\varepsilon}

\renewcommand{\phi}{\varphi}

\newcommand{\co}{\colon}
\newcommand\norm{{\|\cdot\|}}

\DeclareMathOperator{\cl}{closure}
\DeclareMathOperator{\conv}{conv}

\newtheorem{theorem}{Theorem}
\newtheorem{lemma}{Lemma}[section]

\newtheorem*{problem}{Problem}

\theoremstyle{remark}
\newtheorem{remark}[lemma]{Remark}

\theoremstyle{definition}
\newtheorem{definition}[lemma]{Definition}

\begin{document}

\title{Uniform approximation of metrics by graphs}

\author{Dmitri Burago}                                                          
\address{Dmitri Burago: Pennsylvania State University,                          
Department of Mathematics, University Park, PA 16802, USA}                      
\email{burago@math.psu.edu}                                                     
                                                                                
\author{Sergei Ivanov}
\address{Sergei Ivanov:
St.Petersburg Department of Steklov Mathematical Institute,
Russian Academy of Sciences,
Fontanka 27, St.Petersburg 191023, Russia}
\email{svivanov@pdmi.ras.ru}

\thanks{The first author was partially supported                                
by NSF grant  DMS-1205597.
The second author was partially supported by
RFBR grant 11-01-00302-a.}



\begin{abstract}
We say that a metric graph is uniformly bounded if the degrees of all vertices are uniformly
bounded and the lengths of edges are pinched between two positive constants; a metric
space is approximable by a  uniform graph if there is one within a finite Gromov-Hausdorff
distance. We show that the Euclidean plane and Gromov hyperbolic geodesic spaces with
bounded geometry are approximable by uniform graphs, and pose a number of open problems.  
\end{abstract}

\maketitle

\section{Introduction}

In this paper we are concerned with graph approximations of Riemannian manifold as metric spaces. 
We will address problems of spectral approximation elsewhere. 

By a {\em metric graph} we mean an undirected graph whose edges are labeled by positive numbers
called {\em edge lengths}.
This naturally turns the set of vertices of the graph into a metric space
(where some distances may be infinite).
Namely one defines the length of a path in a metric graph as the sum of edge lengths
along the path, and then the distance $d_\Gamma(p,q)$ between vertices $p$ and $q$
of a metric graph $\Gamma$ is defined as the infimum of lengths of paths
connecting $p$ and~$q$.

We say that a metric graph is {\em uniform} if there are positive
constants $M$, $D$ and $\delta$ such that the degree of
every vertex is no greater than $M$ and the length of every edge is between $\delta$ and $D$.

We say that a metric space is is approximable by a uniform graph if there exists a uniform graph 
which is within finite Gromov--Hausdorff distance from the space.

For this paper, the reader does not need to even know what the Gromov--Hausdorff distance is.
We say that two metrics on the same set are {\em additively close} if there exists a constant $C$ such that the
difference of the two distances between every two points is at most $C$.
Also recall that a \textit{net} in metric space is a subset which is an $\ep$-net for some $\ep>0$
(for example $\Z^2$ is a net in $\R^2$). Now the definition can be reformulated in a more concrete way 
by means of the following trivial lemma:

\begin{lemma} 
A manifold is approximable by a uniform graph if and only if there exists a uniform graph whose 
vertices form a net 
in the manifold and whose distance function on the set of vertices is additively close to the 
restriction of  the distance function on the manifold.
\qed
\end{lemma}

The general question which remains widely open is the following problem.

\begin{problem} 
What complete (Riemannian or even Finsler) manifolds are approximable by uniform graphs? 
\end{problem}

Let us emphasize that we are interested in an approximation with a bounded
{\it additive} error.   Say, for $\R^2$, the question is trivial for quasi-isometries. Indeed, the standard grid
%is asymptotically $\ell_1$ (that is, its cone at infinity is $\ell_1$).  Hence the grid
approximates the Euclidean metric up to a factor of $\sqrt 2$.
The question of making the same cone at infinity as that of the Euclidean plane
is more subtle
though it is still much easier (see~\cite{PPS} and references there). However let us
reiterate that we want the difference between distances be bounded.

We do not have a single example of a manifold with bounded geometry
 for which we can prove that it is 
not approximable by a uniform graph (one can easily construct such examples with
sectional curvature rapidly going to $-\infty$). However, so far we cannot even prove that $\R^3$ 
with its standard metric is approximable by a uniform graph.

If the answer is affirmative for a certain class of manifolds, one also wonders how the constants in the
definition of the uniform graph and the additive error $C$ depend on the class (dimension, injectivity 
radius and such). This makes this problem meaningful for compact manifolds. Formally, one can 
consider a disjoint union of manifolds from a certain class and ask if it is approximable by a uniform
graph. For instance, asking if the spheres are approximable by uniform graphs with the same 
constants is the same as asking if the disjoint union of spheres with integer radii is approximable 
by a uniform graph.

One can easily see that, if $\R^n$ is approximable by a uniform graph
then the graph can be chosen so that its vertices form the standard integer 
lattice and the degree of every vertex does not exceed three. 
%This sounds like a trivial statement, however there are some hidden difficulties.
We do not use this fact in the sequel,
so we give only a punchline of a (rather easy) argument. 
Modifying the graph to move
vertices to a lattice and reducing their degrees to at most three can be done in three steps. 
First, one can get rid of vertices of higher degree by replacing their neighborhoods by graphs
with controlled parameters. For each vertex, one just cuts all outgoing edges in the middle, removes the vertex 
and replaces it by an appropriate degree~3 metric graph.
It is crucial to make sure that all distances are exactly preserved.
Next, one chooses a very fine square lattice and moves each vertex to the nearest lattice
point. Lengths of edges stay the same.  Still, many points of the lattice are not occupied by vertices. 
To fix this, it is enough to add 
a whole bunch of vertices of degree two on existing edges and move them to lattice points, having made the
edges to pass through all lattice points. 
Finally, one rescales the lattice to get the integer one and multiplies the edge lengths
by the same factor.
Of course, the constant would deteriorate drastically, but 
one still gets a uniform approximation.

The main result of the paper is the following:

\begin{theorem}
\label{main}
There exists a
uniform metric graph $\Gamma$ whose set of
vertices is the standard lattice $\Z^2\subset\R^2$,
such that its distance function $d_\Gamma$ is
additively close to the standard Euclidean metric:
$$
 \bigl|d_\Gamma(p,q) - |p-q|\bigr| \le C
$$
for some constant $C>0$ and all $p,q\in\Z^2$.
\end{theorem}

Even though we do not know the answer to the problem even for $\R^3$ 
(the next paragraph partially explains a difficulty),
the problem becomes much easier for Gromov hyperbolic spaces:

\begin{theorem}
\label{t:hyperb}
Every simply connected complete Riemannian manifold $M$ whose sectional
curvature is negative and bounded away from 0 and $-\infty$
(more generally, every Gromov hyperbolic geodesic space
of bounded coarse geometry) is approximable by a uniform graph. 
\end{theorem}

To illustrate some difficulties arising in approximating even $\R^2$ let us consider
approximations by a periodic graph, that is by a graph $\Gamma$ invariant under 
two integer
translations $(x,y) \mapsto (x+m,y)$ and $(x,y) \mapsto (x,y+n)$. One can think
of constructing such a graph as it first choosing a graph inside a large rectangle 
and then repeating it periodically to tile the entire plane. In this case, one can show 
that there exists a 
norm $\| \cdot \|$ on $\R^2$  such that the distance $d_\Gamma$
is additively close to the norm: there is a  constant $C$ such that
for every two vertices $p,q$ of the graph, $|d_\Gamma(p,q)-\|p-q\|| < C$.
Furthermore, one can see that the unit ball of the norm is a polygon
(with polynomially many sides). Hence a periodic graph not only cannot be
additively close to the Euclidean plane, but actually the difference between its
metric and the Euclidean one grows linearly with the distance between points.

The rest of the paper is organized as follows.
In Section~\ref{construction} 
we derive Theorem~\ref{main} from
analytic lemmas proven in Section \ref{s:fourier}. The proof of 
Theorem \ref{t:hyperb} is contained
in Section \ref{s:hyperb}.
Section~\ref{problems} is not needed for understanding the proofs.
It informally discusses several problems in Dynamics and Analysis motivated by
the proof of Theorem \ref{main}. In particular, resolving some of these problems could
possibly help to handle dimensions higher than two. 

\begin{remark} 
There is a problem that sounds rather similar. It asks if one can 
approximate the Euclidean distance function between points of the integer lattice
in the plane by connecting them by edges of unit length.
It is easy to see that this is equivalent to approximating $\R^2$
by a uniform graph whose edge length are integers.
We heard about this problem from Bruce Kleiner.
Apparently it goes back to Erd\H os, see~\cite{PPS}.
This is however a rather different question due to its Number Theory aspects.
Nonetheless in the hyperbolic case (Theorem \ref{t:hyperb})
our construction is very robust and it is easy to see that all edges can be assigned integral,
see Remark~\ref{r:hyperb-integer} at the end of Section \ref{s:hyperb}.
\end{remark}

\noindent \textit{Acknowledgement}.
We are grateful to an anonymous referee for thorough reading our paper, 
very useful suggestions and for finding a few typos, especially in formulas. 

%{\it Convention}: We use the same letter $C$ to denote constants used in
%upper bounds, so each new appearance of $C$ can be regarded as taking the 
%maximum of all constants defined above.

\section{Analytic Problems Motivated by the Proof}
\label{problems}

In this section we discuss several problems that emerged from the proof of Theorem \ref{main}
(and more specifically of the analytic lemmas in Section \ref{s:fourier}) and our attempts to generalize it to higher
dimensions. The problems have to do with uniformly distributed sequences and approximating
integrals of functions from certain classes by finite sums of their values along an infinite
sequence.

First of all, up to minor nuances, the key analytic Lemma \ref{l:approx-by-mod}
tells us the following. Given a smooth convex function $f\co[0,1] \to \R$ 
with appropriate boundary conditions, there exists a sequence $\{x_i\}$, $x_i \in [0,1]$ and a constant $C$ such that
$\left|\sum_{i=1}^n |x-x_i| -nf(x)\right| < C$ for all $n \in N$ and $x \in [0,1]$. This means that not only 
averages of distance functions to $x_i$'s converge to $f$ uniformly, but also that the convergence
is extremely fast. This suggests to consider a similar question in higher dimensions, however it is
even unlikely that all smooth convex functions on a square (or a disc) can be approximated by
averages of distance functions to a sequence of points, needless to say that obtaining so fast
approximations is hardly possible. There might be however a reasonable class of functions which
admit such approximations. Actually, for higher-dimensional generalizations of Theorem \ref{main}
one would probably need to average not distance functions to points but rather translations of 
somewhat different functions such as piecewise linear ones.
This circle of problems seems to be widely open.

Furthermore, if one looks into the ``guts" of the proof of Lemma \ref{l:approx-by-mod}, it becomes
clear that it is closely related to a whole line of research which perhaps starts from 
Corput's Conjecture proven by Aardenne--Ehrenfest and further advances by K.~Roth, W.~Schmidt and many others. 
There is an excellent account of this topic, including historical remarks, in \cite{BC}, 
so we refer the reader to this book for all detail. Theorems of Aardenne--Ehrenfest, Roth, Schmidt and other 
show that there are no uniformly distributed sequences on $[0,1]$ (there are infinitely many
values of $n$ such that there are two intervals of the same length but the number of visits 
to them by the sequence until the $n$th member differs by at least $c\log n$). This implies that,
unlike the distance functions, characteristic functions cannot be used for very fast
approximations by averages: if one wants to approximate $f(x)=x$ on $[0,1]$ by a sum
$\frac1n\sum_1^n \chi_{[x_i,1]}$, where $\{x_i\}$ is an infinite sequence, then
for every $C>0$ there are infinitely many $n$ 
such that   $\max_{x\in[0,1]}\left|\sum_1^n  \chi_{[x_i,1]}(x) -nf(x)\right| > C$.
Furthermore, the proof of 
Lemma \ref{l:approx-by-mod} is based on approximating the integral of a function $f\co[0,1] \to \R$ 
by averages $\frac1n\sum_1^n f(x_i)$ along some sequence of points $x_i \in [0,1]$.
Let us say that this approximation is super-fast (for a class of functions~$f$) if for all $f$,  $n$
and $x$, $\bigl|\sum_1^n f(x_i) -n \int_0^1 f\bigr| \leq C$.
Non-existence 
of very uniformly distributed sequences by Schmidt et al imply that the class of characteristic
functions of intervals does not admit super-fast approximations of integrals. However, Lemmas
\ref{l:fourier1} and \ref{approx-integral} imply that such an approximation exists for
a class of functions satisfying appropriate
regularity conditions (this class includes the functions $x\mapsto |x-a|$ but not
characteristic functions of intervals).
Hence we wonder: How fast can we approximate
functions in several variables
(from a certain regularity class)
by averages along a sequence?

\section{The construction}
\label{construction}

The goal of this section in to prove Theorem~\ref{main}
modulo a technical lemma (Lemma~\ref{l:existence})
which is proven in the next section.

We divide $\Z^2$ into two lattices $L$ and $L'$ where
\begin{align*}
 L =\{ (i,j)\in\Z^2: \text{$i+j$ is even}\} ,\\
 L' =\{ (i,j)\in\Z^2: \text{$i+j$ is odd}\} .
\end{align*}
Consider a graph whose set of vertices is $L$
and whose edges connect each node $(i,j)$ to its four 
diagonal neighbors $(i\pm 1,j\pm 1)$.
Our plan is to assign lengths to the edges of this graph
so that the resulting metric $d_L$ on $L$ majorizes the Euclidean norm
and is additively close to it on the set of vectors $(x,y)\in\R^2$ such that $|y|\ge|x|$.

Then similarly one can construct an analogous metric graph on $L'$
whose metric $d_{L'}$ is additively close to the Euclidean one on
vectors $(x,y)$ with $|x|\ge|y|$.
By joining each point $(i,j)\in L$ with $(i,j+1)\in L'$ by an edge of
a sufficiently large fixed length 
one gets a desired metric graph $\Gamma$ whose distance function $d_\Gamma$
is additively close to the Euclidean one.

We construct the graph metric $d_L$ on $L$ as follows.
We choose sequences $\{u_j\}_{j\in\Z}$ and $\{v_j\}_{j\in\Z}$ of positive numbers
(bounded away from 0 and $\infty$)
and for every $i,j\in\Z$ assign length $u_j$ to the edge from $(i,j)$ to $(i+1,j+1)$
and length $v_j$ to the edge from $(i,j)$ to $(i-1,j+1)$.
Note that the resulting metric is invariant under horizontal translations
(by even integer vectors).
The sequences $\{u_j\}$ and $\{v_j\}$ are explicit
but the expression is too cumbersome to be presented here.
An important feature of the construction is that
$u_j+v_j$ is a constant independent of~$j$.

In this section we express (almost explicitly) the graph distance $d_L$
via the sequences $\{u_j\}$ and $\{v_j\}$, see Lemma~\ref{l:dL}.
In the next section we deal with the choice of the sequences
and prove estimates (encapsulated in Lemma~\ref{l:existence})
that control the difference between $d_L$ and the Euclidean metric.

We introduce the following notation and terminology.
By $e_1$ and $e_2$ we denote the standard basis vectors:
$e_1=(1,0)$, $e_2=(0,1)$.
The coordinates of a point $p\in\R^2$ are denoted by $x(p)$ and $y(p)$.
For $j\in\Z$, we denote by $S_j$  the horizontal strip
$$
 S_j = \{ p\in\R^2 : j\le y(p)\le j+1 \} .
$$

\begin{definition}
\label{d:metric-realizes-norm}
Let $\|\cdot\|$ be a norm on $\R^2$ and $d$ a metric on~$L$.
We say that $d$ {\em realizes $\|\cdot\|$ between levels $m$ and $n$}
if $d(p,q)=\|p-q\|$ for any $p,q\in L$
such that $y(p)=m$ and $y(q)=n$.
\end{definition}

\begin{definition}
Let $u,v>0$. The {\em rhombus norm} with
parameters $u,v$ is the norm $\|\cdot\|_{u,v}$
defined as follows: for a vector $p\in\R^2$,
$$
\|p\|_{u,v} = u{|p_1|}+v{|p_2|}
$$
where $p_1$ and $p_2$ are
the components of $p$ in the basis made of vectors $(1,1)$ and $(1,-1)$,
that is, $p=p_1\cdot(1,1)+p_2\cdot(1,-1)$.
\end{definition}

In other words, $\|\cdot\|_{u,v}$ is the norm whose unit ball is a rhombus
with vertices $\pm(1/u,1/u)$ and $\pm(1/v,-1/v)$.
%Note that this norm is the maximal one
%among the norms $\|\cdot\|$ 
%satisfying  $\|(1,1)\|=u$ and $\|(1,-1)\|=v$.

Consider a graph metric $d_L$ on $L$ obtained from sequences
$\{u_j\}$ and $\{v_j\}$ as explained above.
Assume that $u_j+v_j=2D$ for all $j$, where $D$
is a constant independent of~$j$.
This assumption implies that for every $p,q\in L$
with $y(p)\ne y(q)$, the distance $d_L(p,q)$ is realized
by a path confined between the horizontal lines through
$p$ and~$q$.
It follows that for every $j\in\Z$, the metric $d_L$
realizes the rhombus norm $\norm_j:=\norm_{u_j,v_j}$
between levels $j$ and $j+1$.
For points $p,q\in L$ lying on the same horizontal line,
we have $d_L(p,q)=D{|x(p)-x(q)|}$.

As the first step, we show that the distances in the graph metric
are the same as the distances
in a certain  metric on $\R^2$.
Namely consider the following length metric $d$ on $\R^2$:
in each strip $S_j$, $j\in\Z$, the metric is the
restriction of the rhombus norm $\|\cdot\|_j$,
and the metric $d$ on $\R^2$ is the metric gluing
of the metrics in the strips.
%The condition that 
%$u_j+v_j=2D$  means that $\|e_1\|_j=D$ for all $j$, where as usual $e_1=(1,0)$.

The {\em metric gluing} of strips is defined as follows. For points $p,q\in\R^2$,
the distance $d(p,q)$ is the infimum of lengths of broken lines
connecting $p$ and $q$. The length of a broken line $\gamma$
is the sum of lengths of its parts $\gamma\cap S_j$,
and the length of each part is measured in the metric
of the respective strip. 
Since $u_j+v_j=2D$ for all $j$, we have $\|e_1\|_j=D$ for all $j$
and hence any two neighboring strips determine the same length
on their common boundary line.

\begin{lemma}
\label{l:same-distance}
Let $d_L$ and $d$ be as above.
Then $d(p,q)=d_L(p,q)$  for every $p,q\in L$.
\end{lemma}

\begin{proof}
Since the length of a horizontal vector is the same
in all norms $\norm_j$, the distance $d(p,q)$
between any two points $p,q\in\R^2$ 
is realized by a broken line whose $y$-coordinate is monotone
and whose internal vertices have integral $y$-coordinates.
In particular, horizontal lines are shortest paths of~$d$,
and this implies that the assertion of the lemma holds if
$p$ and $q$ lie in the same horizontal line.

Now consider $p,q\in L$ with $y(q)=m$ and $y(p)=m+k$ where $k>0$.
The distance $d(p,q)$ is realized by a broken line with vertices
$p=p_0,p_1,\dots,p_k=q$ such that $y(p_j)=m+j$ for all~$j$.
That is,
$$
 d(p,q) = \sum_{j=0}^{k-1} \|p_j-p_{j+1}\|_{m+j} .
$$
If $p_j$ and $p_{j+1}$ belong to $L$, then $\|p_j-p_{j+1}\|_{m+j}=d_L(p_j,p_{j+1})$.
Therefore it suffices to show that the points $p_1,\dots,p_{k-1}$
can be chosen from our lattice~$L$.

We prove this by induction in $k$.
The base $k=1$ is obvious.
For $k\ge 2$, fix points $p_1,\dots,p_{k-1}$ as above.
For every $t\in\R$ and $j=1,\dots,k-1$,
define $p_j(t)=p_j+te_1$.
Let $p_0(t)=p$ and $p_k(t)=q$ for all~$t$.
Consider a function $f\colon\R\to\R_+$
given by
$$
 f(t) = \sum_{j=0}^{k-1} \|p_j(t)-p_{j+1}(t)\|_{m+j} .
$$
Clearly $f(0)=d(p,q)$ is the minimum of $f$.
On the other hand, $f$ is piecewise linear
and its break points occur only if one of
the segments $[p,p_1(t)]$ and $[p_{k-1}(t),q]$
is an edge of $L$. Indeed, in the definition of $f$
all summands but those for $j=0$ and $j=k-1$
are independent of~$t$, and the summands $\|p_0-p_1(t)\|_m$
and $\|p_{k-1}(t)-p_k\|_{m+k-1}$ are piecewise linear in~$t$
with break points corresponding to the diagonal directions.

Since the minimum of a piecewise linear function $f$
is attained at a break point, we can replace
$p_1,\dots,p_{k-1}$ by $p_1(t_0),\dots,p_{k-1}(t_0)$
where $t_0$ is a break point of~$f$ such that $f(t_0)=f(0)$.
Now at least one of the new points $p_1(t_0)$ and $p_{k-1}(t_0)$
belongs to~$L$, and we apply the induction hypothesis
to the distance from $p_1(t_0)$ to~$q$ or from $p$ to $p_{k-1}(t_0)$.
\end{proof}

The next step is to figure out the distances in the metric $d$
glued from strips. To do this, we associate to every norm
$\|\cdot\|$ on $\R^2$ a concave function $h=h_{\|\cdot\|}$
referred to as the {\em dual profile} of the norm.
It turns out that the metric obtained by gluing normed strips
is given by arithmetic averages of  dual profiles.

\begin{definition}
\label{d:dual-profile}
Let $\norm$ be a norm on $\R^2$ and $D=\|e_1\|$.
The {\em dual profile} of $\norm$ is a function
$h=h_{\norm}\colon [-D,D] \to \R$ defined as follows.
Let $B$ be the unit ball of $\norm$ and
and $B^*$ the dual body to $B$, i.e.,
$$
 B^* = \{ v\in\R^2 : \forall p\in B \ \langle v,p\rangle \le 1\} .
$$
By duality, the horizontal width of $B^*$ equals $2D$, that is,
$[-D,D]\times\R$ is the minimal vertical strip containing $B^*$.
We define
$$
 h(\xi) = \sup\{ \eta\in\R : (\xi, \eta)\in B^* \} , \qquad \xi\in[-D,D].
$$
\end{definition}

The definition implies that $B^*$ is enclosed between
the vertical lines $\{ \xi = D \}$, $\{ \xi = -D\}$
and the graphs $\{ \eta = h(\xi) \}$ and $\{ \eta = -h(-\xi) \}$
in the $\xi\eta$-plane.
Therefore, $h$ uniquely determines the norm $\norm$.

\begin{lemma}
\label{l:gluing}
Let a metric $d$ on $\R^2$ be the metric gluing of strips $S_j$, $j\in\Z$,
where each strip $S_j$ is equipped with a norm $\|\cdot\|_j$
such that $\|e_1\|_j=D$  (where $D>0$ is independent of $j$).
Then  for every $m,n\in\Z$ such that $m<n$, $d$ realizes some norm $\norm^{m,n}$
between levels $m$ and $n$ (see Definition \ref{d:metric-realizes-norm})
with $\|e_1\|^{m,n}=D$.
The dual profile $h^{m,n}$ of $\norm^{m,n}$ is given by
$$
 h^{m,n}(\xi) = \frac 1{n-m} \sum_{j=m}^{n-1} h_j(\xi), \qquad \xi\in [-D,D]
$$
where $h_j$ is the dual profile of $\norm_j$.
\end{lemma}

\begin{proof}
First we associate yet another function to an arbitrary norm $\norm$ on $\R^2$.
Namely define $f=f_{\norm}\colon\R\to\R_+$ by $f(x)=\|(x,1)\|$.
Clearly $f_\norm$ is a convex function with linear asymptotics at $+\infty$ and $-\infty$.
More precisely, if $\|e_1\|=D$ then 
\begin{equation}
\label{e:f-asymptotic}
 f(x)\sim f(-x)\sim Dx, \qquad x\to+\infty .
\end{equation}
Conversely, every positive convex function $f$ satisfying
\eqref{e:f-asymptotic} equals $f_\norm$ for some norm $\norm$
such that $\|e_1\|=D$.

Let us express the dual profile $h=h_\norm$ of $\norm$
in terms of $f=f_\norm$.
Fix $\xi\in[-D,D]$.
First we show that
\be
\label{e:hxi-upper-half}
 h(\xi) = \sup \{\eta\in\R : (\xi,\eta) \in (B\cap H_+)^* \}
\ee
where $H_+=\{(x,y) : y\ge 0 \}\subset\R^2$ is the upper half-plane.
By duality we have
$$
 (B\cap H_+)^* = \cl(\conv(B^*\cup H_+^*)) = \cl(\conv(B^*\cup Y_-)) =B^*+Y_-
$$
where $Y_-=\{(0,y):y\le 0 \}$, $\conv$ denotes the convex hull,
and $B^*+Y_-$ is the Minkowski sum of $B^*$ and $Y_-$.
Therefore
$$
 \sup \{\eta\in\R : (\xi,\eta) \in (B\cap H_+)^* \} = \sup \{\eta\in\R : (\xi,\eta) \in B^* \} .
$$
Since the right-hand side equals $h(\xi)$ by definition, \eqref{e:hxi-upper-half} follows.

We rewrite \eqref{e:hxi-upper-half} as follows.
%$$
\begin{align*}
  h(\xi) &= \sup\{\eta:\ \xi x+\eta y \le 1  \text{ for all $(x,y)\in B\cap H_+$}\} \\
  &=\sup\{\eta:\ \xi x+\eta y \le \|(x,y)\| \text{ for all $x\in\R,y\ge 0$}\} \\
  &=\sup\{\eta:\ \xi x+\eta y \le \|(x,y)\| \text{ for all $x\in\R,y>0$}\} \\
  &=\sup\{\eta:\ \xi x+\eta \le \|(x,1)\| \text{ for all $x\in\R$}\} \\
  &=\sup\{\eta:\ \xi x+\eta \le f(x) \text{ for all $x\in\R$}\} \\
  &=\inf_{x\in\R} \{f(x)-\xi x\} .
\end{align*}
%$$
Here we subsequently use the definition of duality, positive homogeneity of $\norm$,
the fact that  if $y=0$ then $\xi x+\eta y=\xi x\le Dx=\|(x,0)\|$,
again the  positive homogeneity of $\norm$, the definition of $f=f_\norm$,
and the definition of infimum.
Thus
\be
\label{e:legendre}
 h(\xi) = \inf_{x\in\R} \{ f(x) - \xi x \} .
\ee

Now we proceed with the proof of the lemma.
Without loss of generality we assume that $m=0$.
Define $f_j=f_{\norm_j}$.
The distance between points $p=(a,0)$ and $q=(b,n)$,
where $a,b\in\R$, is given by
$
 d(p,q) = g(b-a)
$
where $g\colon\R\to\R_+$ is a function defined by
$$
 g (x) = \inf_{\{x_j\}} \left\{ \sum_{j=0}^{n-1} f_j(x_j) : \text{$\{x_j\}$ such that $\sum_{j=0}^{n-1} x_j= x$} \right\} .
$$
Indeed, to get from $(a,0)$ to $(a+x,n)$ one has to traverse the strips
$S_j$, $j=0,\dots,n-1$, so that the total displacement in the horizontal direction equals $x$.

Define $f(x)=\frac1n g(nx)$. It is easy to see that
the function $f$ is convex and $f(x)\sim f(-x)\sim Dx$ as $x\to+\infty$.
Therefore $f=f_{\norm}$ for some norm $\norm$ such that $\|e_1\|=D$.
By definition, $d$ realizes $\norm$ between levels $0$ and $n$.
It remains to prove that the dual profile $h=h_\norm$
satisfies $h =\frac1n \sum h_j$.

Let $\xi\in[-D,D]$. By \eqref{e:legendre}, we have
$
 h(\xi) = \inf_{x\in\R} \{ f(x)-\xi x\}
$.
Plugging in the definition of $f$ yields
%$$
\begin{align*}
 h(\xi) &=\inf_{x\in\R} \left\{ \inf_{\{x_j\}} 
  \left\{  \frac1n \sum_{j=0}^{n-1} f_j(x_j) : \text{$\{x_j\}$ such that $\sum_{j=0}^{n-1} x_j= nx$} \right\} -\xi x \right\} \\
 &= \inf_{x\in\R} \inf_{\{x_j\}} 
  \left\{ \left( \frac1n \sum_{j=0}^{n-1} f_j(x_j) \right) -\xi x : \text{$\{x_j\}$ such that $\sum_{j=0}^{n-1} x_j= nx$} \right\}  \\
 &=  \frac1n \inf_{x\in\R} \inf_{\{x_j\}} 
  \left\{ \sum_{j=0}^{n-1} f_j(x_j)  -\xi  \sum_{j=0}^{n-1} x_j : \text{$\{x_j\}$ such that $\sum_{j=0}^{n-1} x_j= nx$} \right\}  \\
 &=  \frac1n \inf_{x\in\R} \inf_{\{x_j\}} 
  \left\{ \sum_{j=0}^{n-1} \big(f_j(x_j)  -\xi  x_j\big) : \text{$\{x_j\}$ such that $\sum_{j=0}^{n-1} x_j= nx$} \right\}  \\
 &=  \frac1n \inf_{\{x_j\}} 
  \left\{ \sum_{j=0}^{n-1} \big(f_j(x_j)  -\xi  x_j\big) : x_0,\dots,x_{n-1}\in\R \right\}  \\
 &=  \frac1n 
   \sum_{j=0}^{n-1} \inf_{x_j\in\R} \left\{f_j(x_j)  -\xi  x_j\right\}  = \frac1n \sum_{j=0}^{n-1} h_j(\xi) .
\end{align*}
%$$
The lemma follows.
\end{proof}

Now we return to our special case when each $\norm_j$ is a rhombus norm
$\norm_{u_j,v_j}$ with $u_j+v_j=2D$.
A direct computation shows that the dual profile $h_j:=h_{\norm_j}$
has the form
$
%\label{e:h-by-alpha}
 h_j(\xi) = D - |\xi-\beta_j|
$
where $\beta_j=\frac{u_j-v_j}2$.
By combining Lemma \ref{l:same-distance} and
Lemma \ref{l:gluing} we get the following:

\begin{lemma}
\label{l:dL}
Let a graph metric $d_L$ on our lattice $L$ be defined
as above, using sequences $\{u_j\}$ and $\{v_j\}$
such that $u_j+v_j=2D$ for all $j$.
Then for any two points $p,q\in L$ with $y(p)=m$ and $y(q)=n$
where $m<n$, one has
$$
 d_L(p,q) = \|p-q\|^{m,n}
$$
where $\norm^{m,n}$ is a norm on $\R^2$
whose dual profile $h^{m,n}$ is given by
$$
 h^{m,n}(\xi) = \frac1{n-m} \sum_{j=m}^{n-1} \big(D - |\xi-\beta_j|\big)
$$
where $\beta_j=\frac{u_j-v_j}2$.

In addition, for $p,q\in L$ with $y(p)=y(q)$, one has
$d_L(p,q)=D|x(p)-x(q)|$.
\qed
\end{lemma}

Note that for every $D>0$ and $\beta_j\in(-D,D)$
there exist positive $u_j$ and $v_j$ with $u_j+v_j=2D$
and $\frac{u_j-v_j}2=\beta_j$.
Namely $u_j=D+\beta_j$ and $v_j=D-\beta_j$.
Therefore rather that operating with the sequences $\{u_j\}$
and $\{v_j\}$, one can work with a sequence $\{\beta_j\}$
in the interval $(-D,D)$ after having fixed a constant $D>0$.
The resulting metric graph is uniform if and only if
$\{\beta_j\}$ is separated from $\{D,-D\}$.

From now on we fix $D=\sqrt2$ and introduce a function
$h^0\colon[-D,D]\to\R$ by
\be
\label{e:defh0}
  h^0(\xi) = \begin{cases}
   1-\sqrt{1-\xi^2}, & |\xi|\le\frac{\sqrt2}2 , \\
   \sqrt2-|\xi|, &\frac{\sqrt2}2\le|\xi|\le\sqrt2 .
  \end{cases}
\ee
Clearly $h^0$ is $C^1$ smooth and is the dual profile
of a norm $\norm^0$ given by
$$
 \|(x,y)\|^0= \max \{ |(x,y)| , \sqrt2|x| \}
 = \begin{cases}
  |(x,y)|, & |x|\le|y| , \\
  \sqrt2|x|, & |x|\ge|y| ,
 \end{cases}
$$
where
%$|\cdot|$ denotes the standard Euclidean norm,
$|(x,y)|=\sqrt{x^2+y^2}$.

\begin{lemma}
\label{l:existence}
There exist a constant $C>0$ and a sequence $\{\beta_j\}_{j\in\Z}$
such that $\beta_j\in[-D/2,D/2])$ for all~$j$ and for every $m,n\in\Z$ with $m<n$
one has
$$
  \left| (n-m)h^0(\xi) -  \sum_{j=m}^{n-1} \big(D - |\xi-\beta_j|\big)\right| \le C .
$$
\end{lemma}

The proof of this lemma is given in Section \ref{s:fourier}.
The lemma is proved by analytic methods including Fourier series
and the theory of rational approximations.

Applying Lemma \ref{l:dL} to the sequence constructed in
Lemma \ref{l:existence} yields that the dual profiles
$h^{m,n}$ (determining the distances between levels $m$ and $n$
in our metric graph) satisfy the following inequality:
\be
\label{e:hmn-h0}
 \left| h^{m,n}(\xi) - h^0(\xi) \right| \le \frac C{n-m} .
\ee
Let us show that this inequality implies that the distance $d_L$
is additively close to the norm $\norm^0$. We need the following lemma.

\begin{lemma}
Let $\norm^1$, $\norm^2$ be norms on $\R^2$ such that
$\|e_1\|^1=\|e_1\|^2=D$, and let $h^1$, $h^2$ be their dual profiles.
Then for any $(x,y)\in\R^2$ one has
$$
 \left| \|(x,y)\|^1-\|(x,y)\|^2\right| \le |y|\cdot\sup_{[-D,D]} |h^1-h^2| .
$$
\end{lemma}

\begin{proof}
If $y=0$, then $ \|(x,y)\|^1=\|(x,y)\|^2=D|x|$.
Therefore we may assume that $y>0$.
Every norm $\norm$ is expressed via
its dual profile $h=h_\norm$ as follows.
First observe that
$$
 \|(x,y)\| = \sup_{(\xi,\eta)\in B^*} \{ \xi x+\eta y\} 
$$
where $B^*$ is dual to the unit ball of $\norm$.
If $y>0$, the right-hand side equals
$$
 \sup_{\xi\in[-D,D]} \{ \xi x + h(\xi) y \} .
$$
Therefore
$$
\begin{aligned}
\left| \|(x,y)\|^1-\|(x,y)\|^2\right| 
&=\left|\sup_{\xi\in[-D,D]} \{ \xi x + h^1(\xi) y \} - \sup_{\xi\in[-D,D]} \{ \xi x + h^2(\xi) y \} \right| \\
&\le \sup_{\xi\in[-D,D]} \{| h^1(\xi) y-h^2(\xi)y|\} .
\end{aligned}
$$
The lemma follows.
\end{proof}

This lemma and \eqref{e:hmn-h0} imply that
for any points $p=(x_1,m)$ and $q=(x_2,n)$
with $m,n\in\Z$ and $m<n$,
one has
$$
 \left| \|p-q\|^{m,n} - \|p-q\|^0 \right| \le C .
$$
This and Lemma \ref{l:dL} imply that our graph metric $d_L$
is additively close to $\|\cdot\|^0$.
Recall that $\|(x,y)\|^0 \ge |(x,y)|$ for all $(x,y)\in\R^2$
and $\|(x,y)\|^0 = |(x,y)|$ if $|x|\le|y|$.

Thus we have constructed a uniform graph on $L$ 
such that its distance function $d_L$ is additively close to the Euclidean metric
on vectors with $|y| \geq |x|$ and no  less than the Euclidean distance
minus some constant $C$ for all vectors. Similarly, one can construct a graph 
with vertices in $L'$ whose distance is additively close to the Euclidean one on vectors with $|x|\geq |y|$ 
and also no less than Euclidean distance minus $C$ for all vectors. Now we ``glue'' the two
graphs by choosing a sufficiently large number $M$ (e.g., $M=2C+1$ works) and connecting $(i,j)$ with $(i, j+1)$ 
by an edge of length $M$ for all $(i,j) \in L$. One easily sees that the distance function of the resulting 
graph is additively close to the Euclidean one. This completes the proof of Theorem~\ref{main}.

\begin{remark}
\label{pasting}
The above construction is a special case of the following general observation. If one has a finite collection 
of uniform graphs $\Gamma_i$ such the distance function of each $\Gamma_i$ is additively close to metrics of 
some norms $\norm_i$, then there exists a uniform graph $\Gamma$ whose distance function is additively close 
to the metric of the norm whose unit ball is the convex hull of the unit balls of $\norm_i$. This graph
is easily constructed by connecting sufficiently close vertices 
 of $\Gamma_i$'s by edges of a sufficiently 
 large fixed length. 
 \end{remark}

\section{Approximation of functions in one variable}
\label{s:fourier}

\begin{lemma}
\label{l:fourier1}
Let $f\colon S^1=\R/\Z\to \R$ be a piecewise $C^1$ function
with $\int f=0$ and ${V(f')\le M<\infty}$
where $V(f')$ denotes the variation of the derivative $f'$
on $S^1$.
Let $\alpha$ be a quadratic irrational.
Then for every positive integer $n$ and every $x\in S^1$, one has
$$
 \left |\sum_{j=0}^{n-1} f(x+j\alpha) \right| \le C(\alpha)\cdot M
$$
for some constant $C(\alpha)$.
\end{lemma}

\begin{proof}
Consider the Fourier series $f(x)=\sum_{k\in\Z} a_k e^{2\pi i kx}$ for $f$.
Note that ${a_0=0}$ since $\int f=0$.
Since $f'$ is of bounded variation, the Fourier series converges uniformly
and moreover $|a_k|\le M/k^2$, see e.g.\ \cite[Ch.II,\S4]{Z}.
The Fourier series for the sum in the left-hand side of the desired inequality
has the form
$$
  \sum_{j=0}^{n-1} f(x+j\alpha)
   = \sum_{k\in\Z\setminus\{0\}} a_k \sum_{j=0}^{n-1} e^{2\pi i k(x+j\alpha)}
   =   \sum_{k\in\Z\setminus\{0\}} b_k(n) e^{2\pi ikx}
$$
where
$$
   b_k(n) = a_k\sum_{j=0}^{n-1} \left(e^{2\pi i k\alpha}\right)^j
   = a_k \cdot \frac{1-e^{2\pi i k\alpha n}}{1-e^{2\pi i k\alpha}} .
$$
Therefore
$$
 \left |\sum_{j=0}^{n-1} f(x+j\alpha) \right|  \le \sum_{k\in\Z\setminus\{0\}} |b_k(n)e^{2\pi ikx}|
  = \sum_{k\in\Z\setminus\{0\}} |b_k(n)| .
$$
Thus it suffices to prove that
$$
 \sum_{k\in\Z\setminus\{0\}} |b_k(n)| \le C(\alpha)\cdot M
$$
for all $n$.

For $t\in\R$, we denote by $d(t,\Z)$ the distance from $t$ to the nearest integer.
Note that $d(-t,\Z)=d(t,\Z)$.
One easily sees that $|1-e^{2\pi it}|\ge 4d(t,\Z)$ for all $t\in\R$.
Substituting $t=k\alpha$ yields $|1-e^{2\pi ik\alpha}|\ge 4d(k\alpha,\Z)$.
Since $|a_k|\le M/k^2$, it follows that
$$
 |b_k(n)| \le \frac{2a_k}{|1-e^{2\pi ik\alpha}|}  \le \frac{M}{2k^2\cdot d(k\alpha,\Z)}.
$$
It remains to prove that
\begin{equation}
\label{finite-sum}
 \sum_{k=1}^\infty \frac 1{k^2\cdot d(k\alpha,\Z)} < \infty .
\end{equation}
Indeed, the left-hand side depends on $\alpha$ only, so if it is finite then
we can just denote it by $C(\alpha)$ and the lemma follows.

Since $\alpha$ is a quadratic irrational,
Liouville's Approximation Theorem asserts that
$$
 \left|\alpha - \frac pk\right| > \frac{c}{k^2}
$$
for some constant $c=c(\alpha)>0$ and all $p,k\in \Z$. Multiplying this by $k$ we get
\begin{equation}
\label{liouville}
 d(k\alpha,\Z) = \min_{p\in\Z} |k\alpha-p| > \frac ck
\end{equation}
for every positive integer $k$.

%For a positive integer $l$, 
Consider the partition of $\N$ into sets $N_l$, $l=1,2,\dots$,
defined by
$$
N_l = \{k\in\N: 2^{-l-1}<d(k\alpha,\Z)\le 2^{-l}\} .
$$
For every $k\in N_l$, \eqref{liouville} implies that
$$
 k > \frac c{d(k\alpha,\Z)} \ge c\cdot 2^l .
$$
For any distinct $k_1,k_2\in N_l$, we have
$$
 d(k_1\alpha-k_2\alpha,\Z) \le  d(k_1\alpha,\Z)+ d(k_2\alpha,\Z) \le 2^{1-l} .
$$
On the other hand, applying \eqref{liouville} to $|k_1-k_2|$ in place of $k$
yields
$$
 d(k_1\alpha-k_2\alpha,\Z) = d(|k_1-k_2|\alpha,\Z) \ge \frac c{|k_1-k_2|} ,
$$
therefore
%\begin{equation*}
%\label{sparse}
$|k_1-k_2| \ge c\cdot 2^{l-1}$
%\end{equation*}
for any distinct $k_1,k_2\in N_l$.
Thus the $n$th smallest element of the set $N_l$
is bounded below by $cn\cdot 2^{l-1}$, hence
$$
 \sum_{k\in N_l} \frac 1{k^2\cdot d(k\alpha,\Z)}
  \le 2^{l+1}\sum_{k\in N_l} \frac 1{k^2}
  \le  2^{l+1}\sum_{n=1}^\infty \frac 1{(cn\cdot 2^{l-1})^2}
  = \frac 8{2^l c^2}\cdot\frac{\pi^2}{6}
$$
where $\frac{\pi^2}{6}$ denotes the sum of the series $\sum_{n=1}^\infty \frac1{n^2}$.
Summing these inequalities for $l=1,2,\dots$, we get
$$
  \sum_{k\in\N} \frac 1{k^2\cdot d(k\alpha,\Z)} \le \frac{8}{c^2}\cdot\frac{\pi^2}{6} < \infty .
$$
This completes the proof of \eqref{finite-sum} and hence of the lemma.
\end{proof}

\begin{lemma}
\label{approx-integral}
There exists a constant $C>0$ and a sequence $\{\alpha_j\}_{j=-\infty}^\infty$
of points in $[0,1]$ such that the following holds.
For every piecewise smooth function ${f\colon [0,1]\to\R}$
and every $m,n\in\Z$ such that $m<n$, one has
$$
 \left| \sum_{j=m}^{n-1} f(\alpha_j) - (n-m)\int_0^1 f \right| \le C (V_0^1(f') + |f'(0)| + |f'(1)|).
$$
Here $V_0^1(f')$ denotes the variation of $f'$ on $[0,1]$.
\end{lemma}

\begin{proof}
Fix a quadratic irrational $\alpha$ and define
$$
 \alpha_j = 2d(j\alpha,\Z)=\begin{cases}
  2\{j\alpha\} &\text{if } \{j\alpha\}\le 1/2, \\
  2-2\{j\alpha\} &\text{if }\{j\alpha\}\ge 1/2,
 \end{cases}
$$
where $\{j\alpha\}$ denotes the fractional part of $j\alpha$.
We claim that this sequence works.

Let $A=\int_0^1 f$ and $f_0=f-A$, then $\int_0^1 f_0=0$.
Define $g\colon[0,1]\to\R$ by
$$
 g(x) = \begin{cases}
  f_0(2x) &\text{if }0\le x\le 1/2, \\
  f_0(2-2x) &\text{if }1/2\le x\le 1 .
 \end{cases}
$$
Observe that $g(0)=g(1)$. Therefore $g$ (unlike $f$)
descends to a continuous and hence piecewise smooth
function $\bar g$ on the circle $\R/\Z$.
From definitions by term-by-term comparison
we get
$$
  \sum_{j=m}^{n-1} f(\alpha_j) - (n-m)\int_0^1 f
  =\sum_{j=m}^{n-1} f_0(\alpha_j) = \sum_{j=m}^{n-1} g(\{j\alpha\})
 =\sum_{j=m}^{n-1} \bar g(j\alpha) .
$$
Since $\int\bar g=0$, the previous lemma implies that
$$
 \left|\sum_{j=m}^{n-1} \bar g(j\alpha)\right| 
 =\left|\sum_{j=0}^{n-m-1} \bar g(m\alpha+j\alpha)\right|
  \le C(\alpha) V(\bar g')
   = 4C(\alpha) (V_0^1(f') + |f'(0)| + |f'(1)|) .
$$
The lemma follows.
\end{proof}

\begin{lemma}
\label{l:approx-by-mod}
Let $h\colon[a,b]\to\R$ be a smooth function such that $h''>0$
everywhere on $[a,b]$, $h'(a)=-1$, $h'(b)=1$
and $h(a)+h(b)=b-a$.
Then there exists a constant $C$ and a sequence $\{\beta_j\}_{j=-\infty}^\infty$ of points in $[a,b]$
such that for  every $x\in[a,b]$ and every $m,n\in\Z$ such that $m<n$, one has
$$
 \left|\sum_{j=m}^{n-1} |x-\beta_j| -(n-m)h(x)\right| \le C .
$$
\end{lemma}

\begin{proof}
Define a map $\varphi\colon[a,b]\to\R$ by
$$
 \varphi(t) = \frac12\big(h'(t)+1\big) .
$$
The assumptions that $h''>0$, $h'(a)=-1$ and $h'(b)=1$
imply that $\varphi$ is a diffeomorphism from $[a,b]$ onto $[0,1]$.
Let $\beta_j=\varphi^{-1}(\alpha_j)$ where $\{\alpha_j\}$
is the sequence from the previous lemma.
Then, for every $x\in[a,b]$ we can write
$$
 \sum_{j=m}^{n-1} |x-\beta_j| = \sum_{j=m}^{n-1} f_x(\alpha_j)
$$
where $f_x$ is a function on $[0,1]$ defined by
$$
 f_x(y) = |x-\varphi^{-1}(y) |.
$$
Since $\varphi^{-1}$ is smooth, the variations $V_0^1(f_x')$
are bounded above by some constant $C_0$ which is
independent of $x$. Therefore, by the previous lemma,
$$ 
 \left|\sum_{j=m}^{n-1} f_x(\alpha_j)-(n-m)\int_0^1 f_x\right| \le C_1\cdot C_0 =: C
$$
where $C_1$ is the constant from the previous lemma.
To complete the proof, it now suffices to show that
$$
 \int_0^1 f_x = h(x)
$$
for every $x\in[a,b]$.
This is shown by the following computation:
%$$
\begin{align*}
 \int_0^1 &f_x(y)\,dy = \int_0^1  |x-\varphi^{-1}(y) |\,dy \\
 &=\int_a^b {|x-t|}\,\varphi'(t)\,dt \qquad\qquad\text{by substitution $y=\varphi(t)$}\\
 &=\frac12\int_a^b {|x-t|}\,h''(t)\,dt \qquad\qquad\text{by the definition of $\varphi$}\\
 &=\frac12\left(\int_a^x (x-t)h''(t)\,dt + \int_x^b(t-x)h''(t)\,dt\right) \\
 &=\frac12\left(\int_0^x h'(t)\,dt -(x-a)h'(a) - \int_x^b h'(t)\,dt + (b-x)h'(b)   \right) 
 \qquad\text{by parts} \\
 &= \frac12\bigg(  h(x)-h(a) +(x-a) -h(b) + h(x) +(b-x)\bigg)
 =h(x) .
\end{align*}
%$$
The last two equalities employ the assumptions that $h'(a)=-1$, $h'(b)=1$
and $h(a)+h(b)=b-a$.
\end{proof}

Now we are in a position to prove Lemma \ref{l:existence}.

\begin{proof}[Proof of Lemma \ref{l:existence}]
We apply Lemma \ref{l:approx-by-mod}
to the interval $[a,b]=[-\frac{\sqrt2}2,\frac{\sqrt2}2]$
and the function $h$ given by
$$
 h(x) = \sqrt2 - \sqrt{1-x^2}, \qquad t\in[-\sqrt2/2,\sqrt2/2] .
$$
In is easy to see that these data satisfy the assumptions
of Lemma \ref{l:approx-by-mod}.
Hence there exist a constant $C>0$ and a sequence
$\{\beta_j\}\subset[-\frac{\sqrt2}2,\frac{\sqrt2}2]$ such that
$$
\left| (n-m) h(x)-\sum_{j=m}^{n-1} |x-\beta_j|\right| \le C
$$
for all $x\in[-\frac{\sqrt2}2,\frac{\sqrt2}2]$
and all $m,n\in\Z$ such that $m<n$.
Hence
$$
 \left|(n-m) (\sqrt2-h(x))-\sum_{j=m}^{n-1} (\sqrt2-|x-\beta_j|) \right| \le C
$$
or, equivalently,
$$
 \left|(n-m)\sqrt{1-x^2}-\sum_{j=m}^{n-1} (\sqrt2-|x-\beta_j|) \right| \le C
$$
for all $x\in[-\frac{\sqrt2}2,\frac{\sqrt2}2]$
and all $m,n\in\Z$ such that $m<n$.

On the intervals $[\frac{\sqrt2}2,\sqrt2]$ and $[-\sqrt2,-\frac{\sqrt2}2]$,
the functions $x\mapsto \sqrt2-|x-\beta_j|$ are linear with slopes
$-1$ and $1$, respectively.
Therefore
$$
 \left|(n-m) h^0(x)-\sum_{j=m}^{n-1} (\sqrt2-|x-\beta_j|) \right| \le C
$$
for all $x\in[-\sqrt2,\sqrt2]$,
where the function $h^0$ is defined by \eqref{e:defh0}.
\end{proof}

\section{The hyperbolic case}
\label{s:hyperb}

In this section we prove Theorem~\ref{t:hyperb}. Let us indicate that the proof uses
only $\delta$-hyperbolicity, Gromov's Morse Lemma for Gromov hyperbolic spaces
(see \cite{BBI} for basic definitions and the Morse Lemma),
and a very weak corollary of bounded geometry: given any $\ep$, there is an
$\ep$-net such that for every $R>0$ there is a constant $C=C (R, \ep)$  
such that every ball of radius $R$ contains at most $C$ points from the
net. 

We set $\ep=\delta$ and fix such an $\ep$-net $X$.
This is the set of vertices of our graph. 

We construct the
graph in two steps. Fix a point $p\in M$. First, we build a tree such that all 
distances in this tree from $p$ to points in $X$ are (exactly) equal to distances in 
$M$, and distances in $M$ between other pairs of points in $X$ do not exceed
distances in the tree up to an additive constant.

To achieve this, for every point $q \in X\setminus\{p\}$ we choose  a $q' \in X$ such that:
\begin{enumerate}
\item $q'$ lies within distance $\ep$ from the geodesic segment $[pq]$;
\item $ d(q, p) - 15 \ep <d(q', p) < d(q, p) - 5 \ep$ if $d(p,q)>5\ep$;
\item if $d(q, p)\le 5\ep$, then $q'=p$.
\end{enumerate}
The existence of such $q'$ follows simply from the triangle inequality and
the definition of $\ep$-nets.
%(in order to avoid confusing the readers we ought to
%say that the constants are not nearly sharp, they are basically random and are 
%chosen with a lot of room as a reserve).
Note that we choose just one point $q'$ (a ``parent'') for every $q$ in~$X$.

We connect every $q\in X$ to its ``parent'' $q'$ by an edge
and set the length of this edge to be $d(q, p)-d(q',p)$.
The resulting graph $T$ is a tree.
We denote the distance in $T$ by $d_T$.
By construction, $d_T(p,q)=d(p,q)$ for every $q\in X$.
Moreover
$
 d_T(q_1,q_2) = d(q_1,p)-d(q_2,p)
$
for any $q_1,q_2\in X$ such that
$q_2$ lies on the $T$-path from $q_1$ to $p$.

Note that since each point $q'$ is connected to points which are no further 
than $100\ep$ away, the degree of vertices in this tree is uniformly
bounded due to the bounded geometry assumption.

Let us make an important though obvious observation here. For a 
shortest path from $p$ to $q$
in the tree,  consider a broken geodesic line obtained by joining adjacent 
vertices along this 
path by shortest segments in $M$. These broken lines are quasi-geodesics with the 
same quasi-geodesic constant (say, 10) and hence by the Morse Lemma there is a 
constant $D$ such that every shortest path from $q$ to $p$ in the tree stays 
within the $D$-neighborhood of a shortest path in $M$. This implies that 
for any $q_1, q_2\in X$ such that $q_2$ lies on the $T$-path between $q_1$ and $p$,
we have $|d_T(q_1,q_2)-d(q_1,q_2)|\le 2D$.

Now for any two $q_1, q_2\in X$ the $T$-path from $q_1$ to $q_2$
contains a point $q$ which lies on both $T$-paths connecting
$q_1$ and $q_2$ to $p$.
Hence
$$
 d_T(q_1,q_2) = d_T(q_1,q)+d_T(q_2,q) \ge d(q_1,q)+d(q_2,q)-4D \ge d(q_1,q_2)-4D
$$
by the triangle inequality.
Thus the distance between $q_1$ and $q_2$ in $M$ 
cannot exceed that in the tree by more than an additive constant.

For any two $q_1, q_2\in X$, by (one of the definitions
of) $\delta$-hyperbolicity there is a point $p'$ on the $M$-geodesic 
$[q_1,q_2]$ such that the distance from $p'$ to geodesic segments
connecting $q_1$ and $q_2$ to $p$ is less than $\delta$. Thus there are 
points $q_1',q_2'\in X$ such that $q_i'$ lies between $q_i$ and $p$ in $T$ 
and $d(q_i',p) < D_1:=D+100\ep+\delta$ for $i=1,2$.
 
Now we are prepared for Step 2.
We connect by a new edge every pair of points $x,y\in X$
such that $d(x,y)<2D_1$.
By our bounded geometry assumption,
the degrees of  vertices in the graph remain uniformly bounded.   
We set the length of each of the new edges to be $2D_1+4D$.
This guarantees that the distances in the new graph still cannot
be shorter than those in $M$ by more than $4D$.
Now for $q_1,q_2\in X$ consider the points $q_1'$ and $q_2'$
constructed in the previous paragraph. They are connected by
a new edge, hence the distance between $q_1$ and $q_2$ in the new graph 
is bounded above by
$$
 d_T(q_1,q_1')+d_T(q_2,q_2') + 2D_1+4D \le  d(q_1,q_1')+d(q_2,q_2') + 2D_1+12D .
$$
Since $q_1'$ and $q_2'$ lie within distance $\delta$ from a point $p'$ on
the geodesic $[q_1q_2]$, we have $d(q_1,q_1')+d(q_2,q_2')\le d(q_1,q_2)+2\delta$.
Thus that distances in the graph and in $M$ differ by no more than
by an additive constant, namely by $2D_1+12D+2\delta$. 
This finishes the proof of Theorem~\ref{t:hyperb}.

\begin{remark}
\label{r:hyperb-integer}
The construction can be modified in a trivial way to assign integral lengths to all edges. 
First, one sets $\ep$ and $\delta$ larger than say 10. Next, 
at Step 1, we set the length of the edge from $q$ to $q'$ to be $[d(q, p)]-[d(q',p)]$
where $[\cdot]$ denotes the integral part.
At Step 2, we take any integer greater than $2D_1+4D$
and assign this length to all new edges. It is easy to check that the argument goes
through exactly the same way. 
\end{remark}

\end{document}